\documentclass[12pt,a4paper]{article}

\usepackage{latexsym}
\usepackage{amssymb,exscale}
\usepackage[centertags]{amsmath}

\usepackage{amsthm}

\numberwithin{equation}{section}

\newtheorem{theorem}[equation]{Theorem}
\newtheorem{lemma}[equation]{Lemma}
\newtheorem{claim}[equation]{Claim}

\newtheorem{corollary}[equation]{Corollary}

\newtheorem{definition}[equation]{Definition}

\renewcommand{\phi}{\varphi}

\newcommand{\D}{\mathrm{d}}

\renewcommand{\(}{\bigl(}
\renewcommand{\)}{\bigr)\vphantom{)}}

\newcommand{\const}{\operatorname{const}}

\newcommand{\Tra}{\operatorname{Tra}}
\newcommand{\Ra}{\operatorname{Ra}}
\newcommand{\Di}{\operatorname{Di}}
\newcommand{\Const}{\operatorname{Const}}

\renewcommand{\div}{\operatorname{div}}

\newcommand{\R}{\mathbb R}

\newcommand{\Z}{\mathbb Z}
\newcommand{\T}{\mathbb T}

\newcommand{\bB}{\mathbb B}

\newcommand{\One}{{1\hskip-2.5pt{\rm l}}}

\def\emailwww#1#2{\par\qquad {\tt #1}\par\qquad {\tt #2}\medskip}

\title{Uniformly spread measures and vector fields}

\author{Mikhail Sodin\footnotemark[1]\ \ and Boris Tsirelson\footnotemark[1]}

\date{}

\begin{document}

\maketitle

\begin{abstract}

We show that two different ideas of uniform spreading of locally
finite measures in the $d$-dimensional Euclidean space are
equivalent. The first idea is formulated in terms of finite distance
transportations to the Lebesgue measure, while the second
idea is formulated in terms of vector fields connecting
a given measure with the Lebesgue measure.

\end{abstract}

{
\renewcommand{\thefootnote}{\fnsymbol{footnote}}
\footnotetext[1]{%
 Supported by the Israel Science Foundation of the Israel Academy of Sciences
 and Humanities}
}

\section{Introduction}

This text aims to disentangle and make explicit some ideas
implicit in our work \cite{SoTsi}. It can be read independently
of \cite{SoTsi}.

\medskip

Given a locally finite non-negative measure $\nu$ on the Euclidean space
$\R^d$, we are interested to know \emph{how evenly is the measure
$\nu$ spread over $\R^d$}? First, we consider counting measures for discrete
subsets $X\subset\R^d$: $\nu_X = \sum_{x\in X} \delta_x$
where $\delta_x$ is a unit measure sitting at $x$. Following
Laczkovich~\cite{La1, La2},
we say that the set $X$ (and the measure $\nu_X$) are
\emph{uniformly spread in $\R^d$} if there exists a bijection
$S\colon \Z^d \to X$ such that $\displaystyle
\sup \{ |S(z)-z|\colon z\in\Z^d \} <\infty$.
Equivalently, there exists a measurable map $T\colon \R^d\to X$ called the
\emph{marriage} between  the $d$-dimensional Lebesgue measure $m_d$ and $\nu_X$  (a.k.a. ``matching'',
``allocation'') that pushes forward the Lebesgue measure $m_d$ to $\nu_X$ and such that
$\sup \{ {|T(x)-x|\colon} {x\in\R^d} \} <\infty$.

To extend the notion of uniform spreading to arbitrary measures on $\R^d$, we
use the idea of the \emph{mass transfer} that goes back to
G.~Monge and L.~V.~Kantorovich~\cite[Chapter~VIII, \S4]{KA}.
Let $\nu_1$ and $\nu_2$ be locally finite positive
measures on $\R^d$. We call a positive locally finite measure
$\gamma$ on $\R^d\times \R^d$ a \emph{transportation} from $ \nu_1
$ to $ \nu_2 $, if
$\gamma$ has marginals $\nu_1$ and $\nu_2$, that is
\[
\iint_{\R^d\times \R^d} \phi(x)\, \D\gamma(x,y) = \int_{\R^d} \phi(x) \,
\D\nu_1(x)\,,
\]
and
\[
\iint_{\R^d\times\R^d} \phi(y)\, \D\gamma(x,y) = \int_{\R^d}
\phi(y) \, \D\nu_2(y)
\]
for all continuous functions $\phi\colon \R^d\to\R^1$ with a
compact support. Note that if there exists a map $\tau\colon
\R^d\to\R^d$ that pushes forward $ \nu_1 $ to $ \nu_2 $, then
the corresponding transportation $\gamma_\tau$ is defined as follows:
\[
\iint_{\R^d\times \R^d} \psi(x, y)\, \D\gamma_\tau(x,y)
= \int_{\R^d} \psi(x, \tau (x))\, \D\nu_1(x)
\]
for an arbitrary continuous function $\psi\colon \R^d\times \R^d\to\R^1$ with
a compact support.

The better $\gamma$ is concentrated near the
diagonal of $\R^d\times \R^d$, the closer the measures $\nu_1$ and
$\nu_2$ must be to each other. We shall measure such a concentration
in the $L^\infty$-norm and set
\[
\Tra(\nu_1, \nu_2) = \inf_\gamma ||x-y||_{L^\infty(\gamma)} =
\inf_\gamma \, \sup \{|x-y|\colon x,y \in
\operatorname{spt}(\gamma)\} \in [0,\infty] \,,
\]
where the infimum is taken over all transportations  $\gamma$, and
`$\operatorname{spt}$' denotes the closed support.
Clearly, $ \Tra(\nu_1,\nu_2) + \Tra(\nu_2,\nu_3) \ge \Tra(\nu_1,\nu_3) $.
By $\Tra (\nu) =
\Tra(\nu, m_d)$ we denote the transportation distance between the
measure $\nu$ and the Lebesgue measure $m_d$. If $\nu=\nu_X$ with
a discrete set $X\in\R^d$, then
\[
\const\cdot \Tra (\nu) \le \inf_S \sup_{x\in\Z^d}
| S(x) - x | \le \Const\cdot \Tra (\nu)
\]
where the infimum is taken over all bijections $S\colon \Z^d\to X$.
This follows, for instance, from the locally finite marriage lemma
discussed two paragraphs below.
Throughout the paper, `Const' and `const' mean positive constants
that depend only on the dimension $d$. The values of these constants
can be changed at each occurrence.

\medskip

There exists a dual definition of the transportation distance $\Tra (\nu_1,
\nu_2)$.
The distance $\Di(\nu_1, \nu_2)$ is defined as the
infimum of $r\in (0, \infty)$ such that
\begin{equation}\label{eq.dis}
\nu_1 (B) \le \nu_2(B_{+r})\,, \qquad \text{and}\quad  \nu_2(B) \le
\nu_1(B_{+r})\,,
\end{equation}
for each bounded Borel set $B\subset \R^d$. Here, $B_{+r}$ is the closed
$r$-neighbourhood of $B$ (actually, for our purposes, we could take open
neighbourhoods as well). The distance $\Di$ ranges from $0$ to
$+\infty$, the both ends are included. We define the {\em discrepancy}
of the measure $\nu$ as $D(\nu)=\Di (\nu, m_d)$. The following duality is classical.

\begin{theorem}\label{thm.equiv}
$\Tra (\nu_1, \nu_2) = \Di(\nu_1, \nu_2)$. In particular,
$\Tra (\nu) = D(\nu)$.
\end{theorem}

For finite measures $\nu_1$ and $\nu_2$, it
follows from a result of Strassen~\cite[Theorem~11]{Strassen} and Sudakov~\cite{Sudakov}.
If the measures $\nu_1$ and $\nu_2$ are counting measures of discrete sets $X_1$ and $X_2$,
then it follows from a locally finite version of the marriage lemma due to M.~Hall and R.~Rado,
see Laczkovich~\cite{La2}. Note that the locally finite marriage lemma asserts existence
of a bijection between the sets $X_1$ and $X_2$ which is more than a transportation
from $\nu_1$ to $\nu_2$. Theorem~\ref{thm.equiv} is also mentioned in Gromov~\cite[Section~$3\frac12$]{Gromov},
though the exposition there is quite sketchy. For the reader's convenience,
we recall the proof in Appendix.

\medskip
A different idea of connecting the measures $\nu$ and $m_d$ comes
from the potential theory. We say that a locally integrable vector
field  $v$ {\em connects} the measures $\nu$ and $m_d$ if $\div v =
\nu-m_d$ (in the weak sense), that is
\[
\int_{\R^d} \langle v(x), \nabla \phi (x) \rangle\, \D m_d(x) =
-\int_{\R^d} \phi (x) \, \D(\nu - m_d)(x)
\]
for all smooth compactly supported functions $\phi\colon \R^d\to
\R^1$. It is easy to see that such a field always exists. For
instance, we can take $v=\nabla h$, where $h$ is a solution to the
Poisson equation $\Delta h=\nu-m_d$ in $\R^d$ which exists
due to a subharmonic version of Weierstrass' representation
theorem \cite[Theorem~4.1]{HK}.

Let $\bB (x; r)$ be a ball of radius $r$ centered at $x$, and
$r\bB = \bB (0; r)$. Set $ \chi_r =
\frac1{m_d(r\bB)} \One_{r\bB} $
where $\One_{r\bB}$ is the indicator function of the ball $r\bB$.
We measure the size of the field $v$ as follows.

\begin{definition}
For a locally integrable vector field $v$ on $\R^d$, we set
\[
\Ra (v) = \inf_{r>0} \left\{ r + \|\, v*\chi_r\, \|_\infty \right\}
\quad {\rm and} \quad
\widetilde\Ra (v) = \inf_{r>0} \left\{ r + \|\, |v|*\chi_r\, \|_\infty \right\}\,,
\]
where $*$ denotes the convolution.
\end{definition}
Evidently, $\Ra (v) \le \widetilde\Ra (v) \le \|v\|_{\infty} $.
Note that the multiplicative group $\R_+$ acts by scaling on the measures and vector fields:
$\nu_t (B) = \nu(tB)$, $v_t(x) = t^{-1}v(tx)$. These actions are `coordinated': if $\div v = \nu-m_d$,
then $\div v_t = \nu_t - m_d$, and they are respected by our definitions of $\Tra$, $\Ra$
and $\widetilde\Ra$: $\Tra (\nu_t) = t^{-1}\Tra (\nu)$, $\Ra(v_t) = t^{-1}\Ra(v)$
and $\widetilde\Ra (v_t) = t^{-1}\widetilde\Ra (v)$.

\begin{theorem}\label{thm.main}
Let $\nu$ be a non-negative locally finite measure on $\R^d$.
Then
\[
\const\cdot \inf_v \widetilde\Ra (v) \le \Tra (\nu) \le \Const\cdot \inf_v \Ra
(v) \, ,
\]
where the infimum is taken over all vector fields $v$ connecting
the measures $\nu $ and $m_d$.
\end{theorem}

This is the main result of this note. In the proof of the upper bound
we use duality and actually prove that $D(\nu) \le \Const\cdot \Ra (v)$.
For this reason, our technique gives no idea how transportations
$\gamma$ may look like in the case when $\Tra (\nu)$ is finite.

\begin{corollary}\label{cor}
Let $u$ be a $C^2$-function on $\R^d$ such that
$\Delta u = \nu-m_d$. Then
\[
\Tra (\nu) \le \Const \sqrt{\|u\|_\infty}\,.
\]
\end{corollary}

One can juxtapose this corollary with classical discrepancy estimates
due to Erd\H{o}s and Tur\'an and Ganelius. In \cite{Gan} Ganelius
proved that if $\nu$ is a probability measure
on the unit circumference $\mathbb T\subset \mathbb C$,
and $m$ is the normalized Lebesgue measure on $\T$, then
\[
\sup_{I} | \nu(I) - m(I) | \le \Const \sqrt{ \sup_{\mathbb T} U^\nu }\,,
\]
where the supremum is taken over all arcs $I\subset \mathbb T$, and
\[
U^\nu (z) = \int \log|z-\zeta|\, \D\nu (\zeta)
\]
is the logarithmic potential of the measure $\nu$. Since $U^m$ vanishes on
$\mathbb T$,
we can rewrite this as
\[
\sup_{I} | \nu(I) - m(I) | \le \Const \sqrt{ \sup_{\mathbb T} U^{\nu-m} }.
\]
Note the supremum on the right-hand side, not the supremum of the absolute value
as in our result.

\begin{proof}[Proof of Corollary~\textup{\ref{cor}}]
Consider the convolution $u_r=u*\chi_r$. We have
\[
\nabla u_r=u*\nabla \chi_r, \qquad \text{and} \qquad \Delta u_r = \div \nabla
u_r = \nu*\chi_r - m_d.
\]
Noting that $\nabla \chi_r$ is a finite vector measure of total variation
\[
\| \nabla \chi_r \|_1 = \| \nabla \chi_1 \|_1 \cdot r^{-1} = \Const
\cdot r^{-1}\,,
\]
we have
\[
\Ra (\nabla u_r) \le \|\nabla u_r \|_\infty \le \|u\|_\infty \cdot
\|\nabla\chi_r \|_1 = \frac{\Const}{r}\cdot \|u\|_\infty\,,
\]
and
\[
\Tra(\nu) \le \Tra(\nu,\nu*\chi_r) + \Tra(\nu*\chi_r) \le r + \Const\cdot \Ra
(\nabla u_r) \le r + \frac{\Const}{r}\cdot \| u \|_\infty\,.
\]
Choosing $ r=\sqrt{\|u\|_\infty} $\,, we get the result.
\end{proof}

This corollary immediately yields a seemingly more general result
(cf.~\cite[Theorem~4.3]{SoTsi}).

\begin{corollary}\label{cor2}
Let $u$ be a locally integrable function in $\R^d$ such
that $\Delta u = \nu-m_d$ weakly. Then
\begin{equation}\label{eq.gen}
\Tra (\nu) \le \Const\cdot \inf_{r>0} \big\{
r + \sqrt{\|u*\chi_r\|_\infty} \, \big\} \, .
\end{equation}
\end{corollary}

\begin{proof}[Proof of Corollary~\textup{\ref{cor2}}]
Denote by $\widetilde\chi_r$ the $3$-rd convolution power of $\chi_r$
and put $u_r = u*\widetilde\chi_r$. Then
$u_r$ is a $C^2$-function and $\Delta u_r = \nu*\widetilde\chi_r - m_d$.
Since the function $\widetilde\chi_r$ is supported by the ball $3r\bB$,
we have $\Tra(\nu) \le 3r + \Tra(\nu*\widetilde\chi_r)$. Corollary~\ref{cor}
applied to the smoothed potential $u_r$ yields $\Tra(\nu*\widetilde\chi_r) \le
\Const \sqrt{\|u_r\|_\infty}$. At last, note that $\| u_r\|_\infty \le
\|u*\chi_r\|_\infty \cdot \|\chi_r*\chi_r\|_1 = \|u*\chi_r\|_\infty$
completing the argument.
\end{proof}

\section{Proof of Theorem~\ref{thm.main}}

\subsection{The lower bound}

Here, we construct a vector field $v$ that connects the measure
$\nu$ with the Lebesgue measure $m_d$ and such that $\widetilde\Ra (v)\le
\Const\cdot \Tra (\nu)$.

Let $r>\Tra (\nu)$. For any $x,y\in\R^d$ such that $|x-y|\le
r$, there exists a vector field $v_{x,y}$ concentrated on the ball
$\bB\(\frac{x+y}2; r\)$ such that $\div v_{x,y} =
\delta_x-\delta_y$ (as usual, $\delta_x$ is a point measure at $x$
of the unit mass), and
\[
\int_{\R^d} |v_{x,y}(\xi)|\, \D m_d(\xi) \le \Const \cdot r.
\]
(In order to see that such a field $v$ exists, first, consider a special
case $r=1$; then the general case follows by rescaling.)

Now, we take
\[
v = \iint_{\R^d\times \R^d} v_{x,y}\, \D\gamma (x,y)
\]
where the transportation $\gamma $ connects the measures $\nu$ and
$m$, and is concentrated on the set $\{(x,y): |x-y|\le r \}$. Then
\[
\div v = \iint_{\R^d\times\R^d} (\delta_x - \delta_y)\,
\D\gamma(x,y) = \nu - m_d\,,
\]
and for every $z\in \R^d$
\begin{multline*}
\int_{\bB (z; r)} |v(\xi)|\, \D m_d(\xi) \le \iint_{\R^d\times \R^d}
\D\gamma(x,y) \int_{\bB (z; r)} |v_{x,y}(\xi)|\, \D m_d(\xi) \\
\le \Const \cdot r \cdot \iint \D\gamma(x,y)\,,
\end{multline*}
where the latter integral is taken over such $(x,y)$ that
$\bB\(\frac{x+y}2; r\)\cap \bB \(z; r\)\ne \emptyset$, which
implies $|y-z|\le \frac52 r$. Thus,
\begin{multline*}
\int_{\bB (z; r)} |v(\xi)|\, \D m_d(\xi) \le \Const \cdot r \cdot
\iint_{\R^d\times\R^d}
\One_{\bB (z; 5r/2)}(y)\, \D\gamma(x,y) \\
= \Const \cdot r \cdot \int_{\bB (z; 5r/2)} \D m_d(y) \le \Const
\cdot r^{d+1}\,,
\end{multline*}
that is, $\widetilde\Ra (v)\le \Const \cdot r$, q.e.d.

\medskip

Note that in the argument given above,
the Lebesgue measure $m_d$ can be replaced
with any measure $\mu$ satisfying $\mu \le \Const m_d$. The other
inequality $\Tra(\nu) \le \Const\cdot \Ra(v)$ does not permit such a
replacement. Indeed, if $\eta_x$ is a normalized volume within the
unit ball centered at $x$, then for $|x-y|\ge 2$ we have
$\Tra(\eta_x, \eta_y)\ge \const |x-y|$, whereas it is easy to
construct a vector field $v$ connecting the measures $\eta_x$ and
$\eta_y$ with $||v||_{\infty}\le \Const$. Just take $v = (\nabla
E)*(\eta_x-\eta_y)$, $E$ being a fundamental solution for the
Laplacian in $\R^d$.

\subsection{The upper bound}

In what follows, by a unit cube we mean $Q=\prod_{i=1}^d [a_i, a_i+1] $,
$a_i\in \Z$, $1\le i \le d$.
The proof of the upper bound relies on the following.

\begin{lemma}[Laczkovich]\label{lemma.Laczk}
Suppose that for any set $U\subset \R^d$ which is a finite union of the unit
cubes, we have
\begin{equation}\label{eq.1}
| \nu (U) - m_d(U) | \le \rho m_{d-1}(\partial U)
\end{equation}
with $\rho\ge 1$. Then $D(\nu) \le \Const \rho$.
\end{lemma}

In \cite{La2}, Laczkovich proved this lemma for the counting measure $\nu_X$
of a discrete set $X\subset\R^d$. For the reader's
convenience, will recall the proof of this lemma in~\ref{subsect_lemmaL}.

\medskip

Now, the upper bound in Theorem~\ref{thm.main} will readily follow from the
divergence theorem. We need to show that $D(\nu) \le \Const \Ra(v)$.
A simple scaling argument shows that it suffices to consider only the case $\Ra (v) = 1$. Then
there exists $r\le 2$ such that $\| v*\chi_r \|_\infty \le 2$. Note that
$\div(v*\chi_r) = \nu*\chi_r + m_d$.

let $U\subset\R^d$ be a finite union of the unit cubes. Then denoting by $n$
the outward unit normal to $U$, we have
\begin{multline*}
| (\nu*\chi_r)(U) - m_d(U) | = \bigg| \int_U \div(v*\chi_r)\, \D m_d \bigg|
\\ =
\bigg| \int_{\partial U} \langle v*\chi_r, n \rangle\, \D m_{d-1} \bigg| \le
\| v*\chi_r \|_\infty\, m_{d-1}(\partial U) \le 2m_{d-1}(\partial U) \,,
\end{multline*}
whence, by Laczkovich's lemma, $D(\nu*\chi_r) \le \Const$, and finally,
$D(\nu) \le r + D(\nu*\chi_r) \le \Const$.

\section*{Appendix}

\numberwithin{equation}{subsection}

\renewcommand{\thesubsection}{A-\arabic{subsection}}

\setcounter{subsection}{0}  

\subsection{Transportation  supported by a given set}
\label{appendix}

Here, we shall prove a somewhat more general result than
Theorem~\ref{thm.equiv}.
Let $F\subset \R^d\times\R^d$ be a closed symmetric set such that
\begin{equation}\label{eq1.0}
F\cap(\R^d\times B) \qquad\text{is bounded whenever $B$ is bounded.}
\end{equation}
For $U\subset \R^d$, set $U_{+F} = \{x\in\R^d\colon \exists y\in U \;\;
(x,y)\in F \}$.
If $C\subset \R^d$ is a compact set, then the set $C_{+F}$ is compact as well.

\begin{definition}\mbox{}

\noindent{\rm ({\bf i})} $\Tra (F)$ is a set of all pairs $(\nu_1,\nu_2)$ of
locally finite positive measures $\nu_1$, $\nu_2$ on $\R^d$ such
that there exists a transportation $\gamma$ with
$\operatorname{spt}(\gamma)\subset F$.

\noindent{\rm ({\bf ii})} $\Di (F)$ is a set of all pairs $(\nu_1,\nu_2)$ of
locally finite positive measures $\nu_1$, $\nu_2$ on $\R^d$ such that
\[
\nu_1(C) \le \nu_2(C_{+F}) \qquad \text{and} \qquad \nu_2(C) \le \nu_1(C_{+F})
\]
for any compact subset $C\subset \R^d$.
\end{definition}

\begin{theorem}\label{thm.A} For any closed symmetric set $F\subset \R^d\times\R^d$
satisfying \eqref{eq1.0}, $\Tra (F) = \Di(F)$.
\end{theorem}
See also Kellerer \cite[Corollary 2.18 and Proposition 3.3]{Ke84} for a wide
class of non-closed sets $ F $.

\medskip Theorem~\ref{thm.equiv} follows immediately from
Theorem~\ref{thm.A}: just take a closed symmetric set
$F_r=\{(x,y)\in \R^d\times\R^d\colon |x-y|\le r\}$. Then
\[
(\nu_1,\nu_2)\in \Tra(F_r) \iff \Tra (\nu_1, \nu_2)\le r
\]
and
\[
(\nu_1,\nu_2)\in \Di(F_r) \iff \Di (\nu_1, \nu_2) \le r\,.
\]

\begin{proof}[Proof of Theorem~\textup{\ref{thm.A}}]
The inclusion
$\Tra(F)\subset \Di(F)$ is rather obvious:
\[
\nu_1 (C) = \gamma(C\times \R^d) = \gamma(C\times C_{+F}) \le
\gamma(\R^d \times C_{+F}) = \nu_2(C_{+F}),
\]
and the same for the other inequality.

The proof of the opposite inclusion $\Di(F)\subset\Tra(F)$ is
based on duality. Consider a linear space $C_0(\R^d)$
of continuous functions with compact support in $\R^d$ endowed
with standard convergence: $f_n\to f$ in $C_0(\R^d)$ if there is a
ball $B$ such that $\operatorname{spt}(f_n)\subset B$ for all $n$,
and the sequence $f_n$ converges uniformly to $f$. The dual space
of continuous linear functionals $M(\R^d)$ consists of signed
measures of locally finite variation on $\R^d$ with a usual
pairing $ \nu(f) = \int f\, \D\nu$.
If a linear functional $\nu$ on
$C_0(\R^d)$ is positive (e.g. $\nu(f)\ge 0$ whenever the function
$f$ is non-negative), then it is continuous and is represented by
a non-negative locally finite measure. The same facts are true for the
linear space $C_0(F)$ of continuous functions with a compact
support in $F$, and its dual space $M(F)$.

Consider a mapping $\pi\colon M(F)\to M(\R^d)\oplus M(\R^d)$
acting as $\pi \gamma = (\nu_1, \nu_2)$, where $\nu_1$ and $\nu_2$
are the marginals of the measure $\gamma$. The mapping $\pi $ is
well-defined due to our assumption (\ref{eq1.0}). The conjugate
mapping $\pi'\colon C_0(\R^d)\oplus C_0(\R^d) \to C(F)$ is
$\pi'(f,g)(x,y) = f(x)+g(y)$ for $(x,y)\in F$.
Assume, that $(\nu_1, \nu_2)\in\Di(F)$. We need to show that the
pair $(\nu_1, \nu_2)$ belongs to the image of the cone of positive
measures $M_+(F)$ under $\pi$; in other words, that there exists
$\gamma\in M_+(F)$ such that
\begin{equation}\label{eq1.a}
\gamma (\pi'(f,g)) = (\nu_1, \nu_2)(f,g) = \int f\, \D\nu_1 + \int
g\, \D\nu_2.
\end{equation}
We shall check below that condition $(\nu_1, \nu_2)\in\Di(F)$ ensures
that the RHS of (\ref{eq1.a}) defines a positive linear functional on a linear
subspace $L = \pi'(C_0(\R^d)\times C_0(\R^d))$ of $C_0(F)$.
The linear space $C_0(F)$ is
\emph{subordinated} to its linear subspace  $L$; i.e. for any $\phi
\in C_0(F)$ there are functions $f,g$ in $C_0(\R^d)$ such that
\[
|\phi (x,y)| \le f(x)+g(y)\,, \qquad (x,y)\in F\,.
\]
Then by the classical M.~Riesz' extension theorem (see e.g.
\cite[Chapter~II, \S6, Theorem~3]{Gelfand4}) we can extend this
linear functional to a positive linear functional on the whole
space $C_0(F)$.

It remains to check that the linear functional is well-defined and
positive. Assume that it does not hold; i.e. there is a pair of functions $f,
g\in C_0(\R^d)$ such that
\[
f(x)+g(y) \ge 0, \qquad (x,y)\in F\,,
\]
however,
\[
\int f\, d\nu_1 + \int g\, d\nu_2 < 0\,.
\]
Replacing $g$ by $-g$, we get a pair of functions such that
\begin{equation}\label{eq1.b}
f(x) \ge g(y), \qquad (x,y)\in F\,,
\end{equation}
and
\begin{equation}\label{eq1.c}
\int f\, \D \nu_1 < \int g\, \D \nu_2\,.
\end{equation}
Then, by virtue of (\ref{eq1.b}),
\begin{align*}
\{y\colon g(y)\ge t\}_{+F} &\subset \{x\colon f(x)\ge t\}\,, \\
\{x\colon f(x)\le t\}_{+F} &\subset \{y\colon g(y)\le t\}\,.
\end{align*}
Using, at last, condition $(\nu_1, \nu_2)\in\Di (F)$, we get
\begin{align*}
\nu_2 \( \{y\colon g(y)\ge t\}\) &\le \nu_1 \(\{x\colon f(x)\ge t
 \}\)\,, \qquad t>0, \\
\nu_2 \( \{y\colon g(y)\le t\}\) &\ge \nu_1 \(\{x\colon f(x)\le t
 \}\)\,, \qquad t<0.
\end{align*}
Then
\[
\int g\,\D\nu_2 \le \int f\,\D\nu_1
\]
which contradicts \eqref{eq1.c} and completes the proof of the
theorem.
\end{proof}

\subsection{Proof of lemma of Laczkovich}
\label{subsect_lemmaL}

We check that, for any bounded Borel set $V\subset\R^d$,
\begin{align}
\nu (V) &\le m_d(V_{+C\rho}) \,, \label{eq.2a} \\
m_d (V) &\le \nu(V_{+C\rho})\,. \label{eq.2b}
\end{align}

Take $M=\left[ 2\rho d \right]+1$ and denote by $\mathcal Q_M$ the collection
of all cubes of edge length $M$,
\[
Q = \prod_{i=1}^d [ a_i M, (a_i+1)M ]
\]
with $a_i\in\Z$, $1\le i \le d$. Given a bounded Borel set $V$,
consider the cubes $Q_1$, ..., $Q_n$ from $\mathcal Q_M$
that intersect the set $V$, and denote by $Q_i' = 3Q_i$ the cube concentric
with $Q_i$ of thrice bigger size, $1\le i \le n$. Set
\[
A = \bigcup_{i=1}^n Q_i, \qquad B = \bigcup_{i=1}^n Q_i'\,.
\]
We'll need a simple geometric claim.

\begin{claim}\label{claim}
\[
m_{d-1}(\partial A) \le \frac{2d}{M}\, m_d(B\setminus A)\,, \qquad
m_{d-1}(\partial B) \le \frac{2d}{M}\, m_d(B\setminus A)\,.
\]
\end{claim}

\begin{proof}[Proof of Claim~\textup{\ref{claim}}]
First, we consider the boundary of the set $A$: $\partial A = \bigcup_{j=1}^r
F_j$ where $F_j$ is a face of some cube $Q_{i_j}$. By $P_j$ we denote the cube
obtained by reflection of $Q_{i_j}$ in $F_j$; clearly, for all $j$,
$P_j\subset B\setminus A$. Each cube can be listed at most $2d$ times in the
list of cubes $P_1$, ..., $P_r$ (since every $P_j$ cannot have more than  $2d$
neighbours among the cubes $Q_1$, ..., $Q_n$). Thus,
\[
2d m_d (B\setminus A) \ge \sum_{j=1}^r m_d(P_j) = rM^d = M \cdot r M^{d-1} = M m_{d-1}(\partial A)\,.
\]
This gives us the first inequality. To estimate $m_{d-1}(\partial B)$, we note
that $B\setminus A = \bigcup_{j=1}^s R_j$ where $R_1$, ..., $R_s$ are different
cubes from the collection $\mathcal Q_M$, and that $\partial B \subset
\bigcup_{j=1}^s \partial R_j$.
Whence,
\[
m_{d-1}(\partial B) \le \sum_{j=1}^s m_{d-1}(\partial R_j) \le s \cdot 2d
M^{d-1} = \frac{2d}{M}\, sM^d = \frac{2d}{M}\, m_d(B\setminus A)
\]
proving the claim.
\end{proof}

Now, we readily finish the proof of the lemma. We choose a constant $C$ (depending on the dimension
$d$) so big that $B \subset V_{+C\rho}$. Then
\begin{multline*}
\nu (V) \le \nu (A) \le m_d(A) + \rho m_{d-1}(\partial A) \\
\le m_d(A) + \frac{2d\rho}{M} m_d(B\setminus A) \le m_d(B) \le
 m_d(V_{+C\rho}) \,,
\end{multline*}
whence~\eqref{eq.2a}; and
\begin{multline*}
\nu (V_{+C\rho}) \ge \nu (B) \ge m_d(B)-\rho m_{d-1}(\partial B)
 \ge m_d(B) - \frac{2d\rho}{M}\, m_d(B\setminus A) \\
\ge m_d(B)-m_d(B\setminus A) = m_d(A) \ge m_d(V) \,,
\end{multline*}
whence~\eqref{eq.2b}.

\bigskip
\filbreak { \small
\begin{sc}
\parindent=0pt\baselineskip=12pt
\parbox{2.3in}{
Mikhail Sodin\\
School of Mathematics\\
Tel Aviv University\\
Tel Aviv 69978, Israel
\smallskip
\emailwww{sodin@tau.ac.il} {} } \hfill
\parbox{2.3in}{
Boris Tsirelson\\
School of Mathematics\\
Tel Aviv University\\
Tel Aviv 69978, Israel
\smallskip
\emailwww{tsirel@tau.ac.il} {www.tau.ac.il/\textasciitilde
tsirel/} }
\end{sc}
}

\end{document}